\documentclass[12pt,english]{amsart}
                                                        
\usepackage[inner=1in,outer=1in,top=1in,bottom=1in, marginparwidth=.7in,marginparsep=.1in]{geometry}

\usepackage[T1]{fontenc}
\usepackage{amssymb}

\usepackage{colordvi}
\usepackage{graphicx}

\usepackage{paralist}
\usepackage{babel}
\usepackage{amsmath}
\usepackage{amsthm}
\usepackage{tikz}

\newcommand\A{\mathbb{A}}
\newcommand\C{\mathbb{C}}

\renewcommand\P{\mathbb{P}}

\newcommand\R{\mathbb{R}}

\newcommand\sA{\mathcal{A}}
\newcommand\sB{\mathcal{B}}
\newcommand\cI{\mathcal{I}}

\DeclareMathOperator\trace{trace}

\DeclareMathOperator\diag{diag}
\DeclareMathOperator\Span{span}

\theoremstyle{plain}
\newtheorem{Thm}{Theorem}[section]
\newtheorem{Prop}[Thm]{Proposition}
\newtheorem{Cor}[Thm]{Corollary}
\newtheorem{Lemma}[Thm]{Lemma}

\newtheorem*{Thm*}{Theorem}
\newtheorem*{Prop*}{Proposition}
\newtheorem*{Cor*}{Corollary}
\newtheorem*{Lemma*}{Lemma}
\newtheorem*{Conjecture*}{Conjecture}
\newtheorem*{Conjecture4M-4}{The $\mathbf{4M-4}$ Conjecture \cite{BCMN:14}}

\theoremstyle{definition}

\newtheorem{Example}[Thm]{Example}
\newtheorem{Remark}[Thm]{Remark}

\newtheorem*{Constr*}{Construction}
\newtheorem*{Def*}{Definition}
\newtheorem*{Example*}{Example}
\newtheorem*{Remark*}{Remark}

\newcommand{\Omit}[1]{\textcolor{red}{#1}}
\renewcommand{\Omit}[1]{}

\begin{document}
\title[Projections and phase retrieval]{Projections and phase retrieval}
\author{Dan Edidin}
\thanks{The author's research was partially supported by a Simons Collaboration Grant.}
%\date{\today}

\begin{abstract} 
  We characterize collections of orthogonal projections for which it
  is possible to reconstruct a vector from the magnitudes of the
  corresponding projections.  As a result we are able to show that in
  an $M$-dimensional real vector space a vector can be reconstructed
  from the magnitudes of its projections onto a generic collection of
  $N \geq 2M-1$ subspaces. We also show that this bound is sharp when
  $N = 2^k +1$.  The results of this paper answer a number of
  questions raised in \cite{CCPW:13}.
%{\bf {\large Preliminary draft. Do not circulate.}}
 \end{abstract}
\maketitle

\section{Introduction} 
The phase retrieval problem is an old one in mathematics and its
applications. The author and his collaborators \cite{BCE:06, CEHV:15}
previously considered the problem of reconstructing a vector 
from the magnitudes of its frame coefficients. In this paper
we answer questions raised in the paper \cite{CCPW:13} about phase
retrieval from the magnitudes of orthogonal projections onto a
collection of subspaces.

To state our result we introduce some notation. Given a collection 
of proper linear subspaces $L_1, \ldots L_N$ of $\R^M$ we denote by $P_1, \ldots , P_N$
the corresponding orthogonal projections onto the $L_i$. 
Assuming that the linear span of the $L_i$ is all of $\R^M$ then any vector $x$ can be recovered from vectors
$P_1x, \ldots , P_Nx$
since the linear map 
$$\R^M \to L_1 \times L_2 \times \ldots L_N, x \mapsto (P_1x, \ldots , P_Nx)$$
is injective.

When the $P_i$ are all rank $1$ then a choice of generator
for each line determines a frame and the inner products $\langle P_i x, x\rangle$ are the frame coefficients with respect to this frame. 

In this paper we consider the problem, originally
raised in \cite{CCPW:13},  of reconstructing a vector $x$ (up to
a global sign) from the
magnitudes $$||P_1x||, ||P_2x||, \ldots , ||P_N x||$$ of the projection vectors
$P_1x,\ldots , P_Nx$.

Let $\Phi = \{P_1, \ldots , P_N\}$ be a collection of projections
of ranks $k_1, \ldots , k_N$. Define a map
$\sA_\Phi \colon (\R^M\smallsetminus \{0\})/\pm 1 \to \R_{\geq 0}^N$
by the formula
$$x \mapsto \left(
\langle P_1x,P_1x \rangle, \ldots , \langle P_Nx, P_Nx\rangle\right)$$
As was the case for frames, phase retrieval by this collection 
of projections is equivalent to the map $\sA_\Phi$ being injective.

In \cite{CCPW:13}, Cahill, Casazza, Peterson and Woodland proved that
there exist collections of $2M-1$ projections which allow phase
retrieval. They also proved that a collection $\Phi = \{P_1, \ldots ,
P_N\}$ of projections admits phase retrieval if and only if for every
orthonormal basis $\{\phi_{i,d}\}_{d=1}^{k_d}$ of the linear subspace
$L_i$ determined by $P_i$ the set of vectors $\{\phi_{i,d}\}_{i=1,d=1}^{N,\;k_d}$
allows phase retrieval.
 
Our first result is a 
a more intrinsic 
characterization of collections of projections for which $\sA_\Phi$
is injective.

\begin{Thm} \label{Thm.projectionchar} The map $\sA_\Phi$ is injective
  if and only if for every non-zero $x \in \R^M$ the vectors $P_1x,
  \ldots , P_Nx$ span an $M$-dimensional subspace of $\R^N$, or
  equivalently the vectors $P_1x, \ldots, P_Nx$ form an $N$-element
  frame in $\R^M$.
\end{Thm}

As a corollary we obtain the following necessity result.

\begin{Cor} \label{cor.projection}If $N \leq 2M-2$ and at least $M-1$ of the $P_i$ have rank one, or 
if $N \leq 2M-3$ and at least  least $M-1$ 
of the $P_i$ have rank $M-1$ then $\sA_\Phi$ is not injective.
\end{Cor}

\begin{Remark}
  We will see below that when the $P_i$ all have rank one the
  condition of the theorem is equivalent to the corresponding frame
  having the finite complement property of \cite{BCE:06}
\end{Remark}

Using the characterization of Theorem \ref{Thm.projectionchar}
we show that when $N \geq 2M-1$ 
any generic collection of projections admits phase retrieval. 
Note that this bound of
$2M-1$ is the same as that obtained in \cite{BCE:06}.
\begin{Thm}\label{Thm.projectionsuff}
If $N \geq 2M-1$, then for a generic collection
$\Phi = (P_1, \ldots, P_N)$ of ranks $k_1, \ldots ,k_N$
with $1 \leq k_i \leq M-1$,
 the map $\sA_\Phi$ is injective.
\end{Thm}
\begin{Remark}
By generic we mean that $\Phi$ corresponds to a point in a non-empty
Zariski open subset of a product of real Grassmannians (which has the natural
structure as an {\em affine variety}) whose complement has strictly smaller
dimension. As noted in 
\cite{BCMN:14} one consequence of the generic condition is 
that for any continuous probability distribution on this variety, $\sA_{\Phi}$ is injective with probability one. In particular
Theorem \ref{Thm.projectionsuff} implies that phase retrieval 
can be done with $2M-1$ random subspaces of $\R^M$. This answers 
Problems 5.2 and 5.6 of \cite{CCPW:13}.
\end{Remark}

In \cite{BCE:06} it was proved that $N \geq 2M-1$ is a necessary
condition for frames. However we obtain the following necessity result.
This result was independently obtained by Zhiqiang Xu in his
recent paper \cite{Xu:15}.
\begin{Thm} \label{Thm.grunge}
If $M = 2^k+1$ then $\sA_\Phi$ is not injective for any collection with
$N \leq 2M-2$ projections.
\end{Thm}
\begin{Remark}
Xu also constructed an example of a collection of 6 projections in
$\R^4$ which admit phase retrieval, which shows that the bound $N = 2M-1$
 is not in general sharp.
\end{Remark}
\section{Background  in algebraic geometry}
In this section we give some brief background on some facts we will
need from Algebraic Geometry. For a reference see \cite{Har:95} and
\cite[Chapter 1]{Har:77}.
\subsection{Real and complex varieties }
\label{sec:algebra} 
Denote by $\A^n_\R$ (respectively $\A^n_\C$) the affine space of
$n$-tuples of points in $\R$ (resp. $n$-tuples of points in $\C$).
Given a collection of polynomials $f_1, \ldots, f_m \in \C[x_1, \ldots
, x_n]$ let $V(f_1, \ldots , f_m)$ be the algebraic subset of
$\A^n_\C$ defined by the simultaneous vanishing of the $f_i$. When the
$f_i$ all have real coefficients then we denote by $V(f_1, \ldots,
f_m)_\R \subset \A^n_\R$ the set of real points of the affine
algebraic set $V(f_1,\ldots ,f_m)$.

The relationship between the set of real and complex points of an
algebraic set can be quite subtle. For example the 
algebraic subsets of $\A^2_\C$ defined by the equations $x^2 + y^2 =0 $
and $x^2 - y^2 = 0$ are isomorphic, since the complex linear
transformation $(a,b) \mapsto (a, \sqrt{-1} b)$ maps one to the
other. However, $V(x^2 + y^2)_\R$ consists of only the origin while
$V(x^2 - y^2)_\R$ is the union of two lines.

Given an algebraic set $X = V(f_1, \ldots , f_m)$ we define the
Zariski topology on $X$ by declaring closed sets to be the
intersections of $X$ with other algebraic subsets of
$\A^n_\C$. An algebraic set is {\em irreducible} if it is not the union
of proper Zariski closed subsets. An irreducible algebraic set is
called an {\em algebraic variety}. Every algebraic set has a decomposition
into a finite union of irreducible algebraic subsets.

Note that the set of real points of
an algebraic variety need not be irreducible. For example the affine
curve $V(y^2 - x^3 +x)$ is irreducible, but $V(y^2 - x^3 + x)_\R$ is
the disjoint union of two disconnected pieces.

Given a subset $X\subset \A^n_\C$ the ideal, $I(X)$, of $X$ 
is the set of
all polynomials in $\C[x_1, \ldots , x_n]$ that vanish on $X$. 
Hilbert's Nullstellensatz states that
if $X = V(f_1, \ldots , f_n)$ then $I(X)$ is the radical 
of the ideal generated $f_1, \ldots , f_r$. A variety is irreducible
if and only if $I(X)$ is a prime ideal. A key property of irreducible
algebraic sets is that every non-empty Zariski open set is dense.

\subsubsection{Homogeneous equations and projective algebraic sets}
Denote by $\P^n_\R$ (resp. $\P^n_\C$) the real (resp. complex) projective space
obtained from  $\R^{n+1} \smallsetminus \{0\}$
(resp. $\C^{n+1} \smallsetminus \{0\}$) by identifying 
$(a_0, \ldots , a_n) \sim (\lambda a_0, \ldots , \lambda a_n)$ 
for any non-zero scalar $\lambda$.

Any collection of homogeneous polynomials
$f_1, \ldots , f_m \subset \C[x_0, \ldots , x_n]$
defines a projective algebraic set $X = V(f_1, \ldots , f_r)$.
When the polynomials $f_1, \ldots , f_r$ have real
coefficients then we again let
$V(f_1,\ldots , f_r)_\R$ denote the real points of
$X$. 

As in the affine case we can define the Zariski topology
on a projective algebraic set $X$ by declaring the intersection of $X$
with another projective algebraic set to be closed. An irreducible projective
algebraic set is called a {\em projective variety}. If $X \subset \P^n$
then we define $I(X)$ to be the ideal generated by all homogeneous polynomials
vanishing on $X$. A projective algebraic set is irreducible if and only
if $I(X)$ is a homogeneous prime ideal.

A subset of $\P^n$ is called {\em quasi-projective} if it is a Zariski open
subset of a projective algebraic set. 
Since $\A^n_\C$ is the complement of the hyperplane $V(x_0) \subset \P^n$,
any affine algebraic set is quasi-projective. Following 
\cite[Section I.3]{Har:77} we will use the term {\em variety}
to refer to any affine, quasi-affine (open in an affine), quasi-projective
or projective variety.

\subsubsection{Dimension of a  complex variety}
The dimension of  an algebraic set is most naturally a local invariant.
However, because
varieties are irreducible,
the local dimensions are constant. There are several equivalent
definitions of the dimension of a variety $X$:\\

(i) (Krull dimension) The length of the longest descending chain of proper, irreducible Zariski closed subsets of $X$.\\

(i') If $X \subset \A^n$ is affine then (i) is equal to the length
of the longest ascending chain of prime ideals in the coordinate
ring, $\C[x_1, \ldots, 
x_n]/I(X)$ of $X$.\\

(ii) The transcendence dimension over $\C$ of the field of rational functions
on $X$.\\

(iii) The dimension of the analytic tangent space to a general point
of $X$. (This definition uses the fact that a complex variety contains a dense
Zariski open complex submanifold.)\\

Since an arbitrary algebraic set $X$ can decomposed into a finite union
of irreducible components we can define $\dim X$ to the be the maximum
dimension of its irreducible components.

In the proof of Theorem \ref{Thm.projectionsuff} we will make use
of several facts in dimension theory.

\begin{Thm*}(Krull's Hauptidealsatz \cite[Chapter I, Theorem 1.11A]{Har:77})
Let
$X \subset \A^n$ is an affine variety of dimension $d$.
If $f \in \C[x_1, \ldots , x_n]$
is any polynomial. then $ X \cap V(f)$ is either empty,
all of $X$, or every irreducible component 
of $X \cap V(f)$ has dimension exactly $d-1$. 
\end{Thm*}

\begin{Thm*}(Semi-continuity of fiber dimension \cite[Theorem 11.12]{Har:95})
Let $f \colon X \to Y$ be a morphism of varieties. For any $p \in X$
let $\mu(p) = \dim f^{-1}(f(p))$. Then $\mu(p)$ is an upper-semicontinuous function in the Zariski topology on $X$ - that is, for
any $m$ the locus of points $p \in X$ such that $\dim(f^{-1}(f(p))) \geq m$
is closed in $X$. Moreover, if $\mu$ is the minimum value of $\mu(p)$
then $\dim X = \dim f(X) + \mu$.
\end{Thm*}

\subsubsection{The dimension of the set of real points of a variety}
If $X$ is a variety defined by real equations then we can also define
the dimension of $X_\R$ as a subset of $\A^n_\R=\R^n$ (or
$\P^n_\R$). When $X_\R$ is smooth we can take its dimension as a
manifold. For general $X$, a result in real algebraic geometry
\cite[Theorem 2.3.6]{BCR:98} states that any real semi-algebraic\footnote{A semi-algebraic subset of $\R^n$ is one
  defined by polynomial equations and inequalities. In particular any
  real algebraic set is semi-algebraic.}
subset of $\R^n$ is homeomorphic as a
semi-algebraic set to a finite disjoint union of hypercubes. Thus we
can define $\dim_\R X_\R$ to be the maximal dimension of a hypercube
in this decomposition.

Now if $X \subset \A^n_\R$ is a semi-algebraic set then 
\cite[Corollary 2.8.9]{BCR:98} implies that $\dim_\R X$
equals to the Krull dimension of the algebraic set $V(I(X))$.
As a consequence we obtain the important fact that 
if $f_1, \ldots , f_m$ are real polynomials and $X = V(f_1, \ldots , f_m)$
then $\dim X_\R \leq \dim X$ since $I(X_\R) \supset I(X)$.
\begin{Example}
If $f = x^2 + y^2 \in \R[x,y]$ then $\dim V(f) =1$
but $\dim V(f)_\R = 0$ since $V(f)_\R = \{(0,0)\}$. Note
that in this case $I(V(f)_\R)$ is the ideal $(x,y) \subset \R[x,y]$
and indeed $\dim V(x,y) = 0$ as predicted by \cite[Corollary 2.8.9]{BCR:98}.
\end{Example}

\section{Proof of Theorem \ref{Thm.projectionchar}}
To prove Theorem \ref{Thm.projectionchar} we analyze
the derivative of the map $\sA_\Phi$. Our argument
is similar to an argument used by Murkherjee \cite{Muk:81} to construct
embeddings of complex projective spaces in Euclidean spaces.
Recall that a map $f \colon X \to Y$ of differentiable manifolds is an
{\em immersion} at $x \in M$ if the induced map of tangent spaces $df_x
\colon T_x X \to T_{f(x)} Y$ is injective (so necessarily $\dim X \leq
\dim Y$).

\begin{Lemma}\label{lem.derivcalc}
  Let $P \colon \R^M \to \R^M$ be a rank $k$ projection and let $f
  \colon \R^M \to \R$ be defined by $x \mapsto \langle Px, Px \rangle$.
For any $x \in \R^M$, $df_x(y) = 2\langle Px, y \rangle$ where
  we identify $T_x \R^M = \R^M$ and $T_{f(x)}\R = \R$.
\end{Lemma}
\begin{proof}
Since $P$ is a projection there is an orthonormal 
basis of eigenvectors for $P$. With respect to this basis $P = \diag (1, \ldots 1, 0, \ldots , 0)$ where there are $k$ ones and $M-k$ zeroes. If we choose
coordinates determined by this basis then
$f(x_1, \ldots, x_M) = x_1^2 + x_2^2 + \ldots x_k^2$, so ${\partial f / \partial x_i} = 2x_i$ if $i \leq k$
and ${\partial f / \partial x_i}=0$ if $i > k$.
Thus the derivative at a point $x=(a_1, \ldots , a_M) \in \R^M$
is the linear operator that maps $y=(b_1, \ldots , b_M)$ to 
$2\sum_{i=1}^k a_i b_i = 2 \langle Px, y \rangle$
\end{proof}

\begin{Prop} \label{Prop.immerchar} The map $\sA_\Phi$ is an immersion
  at $\overline{x} \in \left(\R^M\smallsetminus \{0\}\right)/\pm 1$ 
if and only if $P_1x,
  \ldots , P_N x$ span an $M$-dimensional subspace of $\R^M$ where
$x$ is either lift of $\overline{x}$ to $\R^N \smallsetminus \{0\}$.
\end{Prop}
\begin{proof}
Consider the map $\sB_\Phi \colon \R^M \smallsetminus \{0\}\to \R^N$,
$x \mapsto \left(\langle P_1x, P_1x \rangle, \ldots , \langle P_Nx, P_Nx\rangle
\right)$.
The map $\sB_\Phi$ is the composition of $\sA_\Phi$
with the double cover $\R^M \smallsetminus \{0\} \to 
\left(\R^M \smallsetminus \{0\}\right)/\pm 1$.
Since the derivative of a covering map is an isomorphism, it suffices
to prove the proposition for the map $\sB_\Phi$. Applying
Lemma \ref{lem.derivcalc} to each component of $\sB_\Phi$
we see that $d\sB_\Phi$ is the linear transformation
$y \mapsto 2(\langle P_1x, y\rangle, \ldots , \langle
  P_N x, y\rangle)$.  Hence $(d\sB_\Phi)_x$ and thus $(d\sA_\Phi)_x$
is injective if and only
if there is no non-zero vector $y$ which is orthogonal to each
  $P_ix$, or equivalently the vectors $P_ix$ span all of $\R^M$.
\end{proof}

The proof of the theorem now follows from the following proposition.
\begin{Prop}
The map $\sA_\Phi$ is injective if and only if it is a global immersion.
\end{Prop}
\begin{proof}
  First assume that $\sA_\Phi$ is not an immersion. By Proposition
  \ref{Prop.immerchar} there exists an $x \neq 0$ such that $P_1x,
  \ldots , P_Nx$ fail to span $\R^M$. Let $y$ be a non-zero vector
  orthogonal to all the $P_ix$ and consider the vectors $x'= x+y$ and
  $y' = x-y$.  

Then 
$$\begin{array}{ccl}
|P_ix'||^2  & = & \langle P_ix', x'\rangle \; \; \text{since $P_i$ is an orthogonal
projection}\\
& = & \langle P_ix, x \rangle + \langle P_i y, y \rangle + \langle P_iy, x \rangle
+ \langle P_ix, y \rangle\\
& = & ||P_ix||^2 + ||P_i y||^2
\end{array}$$
where the last equality holds because  $$\langle P_iy, x \rangle =
\langle P_iy, P_ix \rangle = \langle P_ix, P_iy \rangle = \langle P_ix , y\rangle = 0.$$

%$\langle P_i x' , P_i x' \rangle = \langle P_x x',  x' \rangle.
%Then $\langle P_ix', x' \rangle = \langle P_ix , x
%  \rangle + \langle P_i y , y\rangle = ||P_ix||^2 + ||P_iy||^2$. (Here
%  we use the fact that $P_i$ is an orthogonal projection so $\langle
%  P_i x , x \rangle = \langle P_ix , P_ix\rangle.$).  

Likewise
  $||P_iy'||^2 =||P_ix||^2 + ||P_i y||^2$.  Hence, either $\sA_\Phi$
  is not injective or $x' = \pm y'$. However, if $x' = \pm y'$ then
  either $x =0$ or $y=0$ which is not the case. Thus $\sA_\Phi$ is not
  injective.

Conversely, suppose that $\sA_\Phi$ is an immersion and suppose that
there exist $x$ and $y$ such that $||P_ix||= ||P_i y||$ for all $i$.
We wish to show that $x = \pm y$. Suppose that $x\neq y$. Then
$x-y \neq 0$. Thus the linear transformation
$(d\sA_\Phi)_{x-y} \colon \R^M \to \R^N$, $z \mapsto \left(\langle P_i(x-y), z \rangle\right)_{i=1}^M $
is injective. On the other hand
$$\langle P_i(x-y),x+y \rangle = \langle P_ix,x \rangle - \langle P_iy,y\rangle
= ||P_i x||^2 - ||P_iy||^2 = 0.$$
(Here we again use the fact that $P_i$ is an orthogonal projection so $\langle P_ix, x\rangle = \langle P_ix , P_ix\rangle$). Hence $x+y=0$, ie $x = -y$.
\end{proof}
\subsection{Proofs of the corollaries}

\begin{proof}[Proof of Corollary \ref{cor.projection}]
Suppose that $P_1, \ldots , P_{M-1}$ have rank 1. Then there is a vector
$x$ such that $P_i x = 0$ for $i =1, \ldots , M-1$, so $P_1x \ldots , P_{M-1}x,\ldots , P_Nx$ cannot span $\R^M$ if $N \leq 2M-2$. Likewise if
$P_1,\ldots , P_{M-1}$ have rank $M-1$ then there exists a vector $y$
such that $P_iy =y$ for $i=1,\ldots , M-1$. In this case
$P_1x \ldots , P_Nx$ fail to span $\R^M$ if $M \leq 2M-3$.
\end{proof}

\begin{Cor}[Complement property \cite{BCE:06}]
If $P_1, \ldots , P_N$ all have rank 1 corresponding to lines
$L_1,\ldots , L_N$ then $\sA_\Phi$ is injective if and only if 
for every partition of $\{1, \ldots , N\}$ into two set
$S, S'$ one of the sets of lines $\{L_i\}_{i\in S}$ or 
$\{L_j\}_{j \in S'}$ spans $\R^M$.
\end{Cor}
\begin{proof}
  Suppose $S \coprod S'$ is a partition of $\{1, \ldots , N\}$ such
  that neither subset of lines $\{L_i\}_{i \in S}$ or $\{L_i\}_{i \in
    S'}$ spans.  Let $x$ be a vector orthogonal to the lines $\{L_i\}_{i \in
    S}$. Thus the span of the vectors $P_ix$ is contained in the span
  of the lines $\{L_j\}_{j \in S'}$ which by assumption do not span
  $\R^M$.

Conversely, suppose that for some $x$ the vectors $P_1x, \ldots , P_Nx$
fail to span $\R^M$. 
Let $S = \{i| P_ix = 0\}$ and let $S' = \{j| P_j x \neq 0\}$.
Since the vectors $\{P_j x\}_{j \in S'}$ are parallel to the lines
$\{L_j\}_{j \in S'}$ we see that these vectors cannot span $\R^M$. On the
other hand the non-zero vector $x$ is orthogonal to each line in the collection
$\{L_i\}_{i \in S}$ so these lines cannot span either.
\end{proof}

\subsection{An example}
We revisit \cite[Example 5.3]{CCPW:13} in the context of Theorem
\ref{Thm.projectionchar}. Let $\{\phi_n\}_{n=1}^3$ and
$\{\psi_n\}_{n=1}^3$ be orthonormal bases for $\R^3$ such that
$\{\phi_n\} \cup \{\psi_n\}$ is full spark (meaning that any 3 element
subset spans).  Since $M = 2+1$ at least 5 projections are required
for phase retrieval by Theorem \ref{Thm.grunge}.  Cahill, Casazza, Peterson and Woodland consider
two collections of subspaces.
$$\begin{array}{rclccrcl}
W_1 &  = & \Span\left(\{\phi_1, \phi_3\}\right)
&  & W_1^\perp & = & \Span\left(\{\phi_2\}\right)\\
W_2 &  = & \Span\left(\{\phi_2, \phi_3\}\right)
&  & W_2^\perp & = & \Span\left(\{\phi_1\}\right)\\
W_3 &  = & \Span\left(\{\phi_3\}\right)
&  & W_3^\perp & = & \Span\left(\{\phi_1, \phi_2\}\right)\\
W_4 &  = & \Span\left(\{\psi_1\}\right)
&  & W_4^\perp & = & \Span\left(\{\psi_2, \psi_3\}\right)\\
W_5 &  = & \Span\left(\{\psi_2\}\right)
&  & W_5^\perp & = & \Span\left(\{\psi_1, \psi_3\}\right)
\end{array}
$$
and showed the collection of orthogonal projections onto
$\{W_i\}_{i=1}^5$ admits phase retrieval while the collection of
orthogonal projections onto $\{W_i^\perp\}_{i=1}^5$ does not.

Using Theorem \ref{Thm.projectionchar} it is easy to see
that the orthogonal projections corresponding to $\{W_i^\perp\}$ do not admit
phase retrieval since the vector $\phi_3$ is orthogonal 
to $W_1^{\perp}, W_2^{\perp}, W_3^{\perp}$. Thus, the images
of the vector $\phi_3$ under the 5 projections cannot span $\R^3$.

Now consider the other collection of orthogonal projections onto
$W_1, \ldots , W_5$  which we denote by
$P_1, \ldots, P_5$.  Since $\{\phi_1, \phi_2, \phi_3, \psi_1, \psi_2,
\psi_3\}$ is full spark the vectors $\{\phi_3, \psi_1, \psi_2\}$ span. Thus
if $x \in \R^3$ is not orthogonal to any of $\phi_3, \psi_1, \psi_2$
then $P_3x, P_4x, P_5x$ span.  If $x$ is orthogonal to $\phi_3$ then
it lies in the plane spanned by $\phi_1$ and $\phi_2$ and is also not
orthogonal to one of $\psi_1$ or $\psi_2$, say $\psi_1$. If $P_5x = 0$
then $x$ is orthogonal to $\psi_2$ which means that it cannot be
orthogonal to either of $\phi_1$ or $\phi_2$ for otherwise $\psi_2$
would have to be parallel to one of the $\psi_i$. It would then follow
that the vectors $P_1x, P_2x, P_4x$ span. If $P_5x \neq 0$ then either
$P_1x, P_4x, P_5x$ or $P_2x, P_4x, P_5x$ span.  If $P_3x \neq 0$ then
$P_1x, P_2x, P_3x$ span if $x$ isn't orthogonal to either $\phi_1$ or
$\phi_2$. If $x$ is orthogonal to $\phi_1$ but not $\phi_2$ then the
vectors $P_2x, P_3x,P_4x, P_5x$ must span.  If $x$ is orthogonal to
both $\phi_1, \phi_2$ then $P_3x, P_4x, P_5x$ span.

\section{Proof of Theorem \ref{Thm.projectionsuff}}
Our proof is similar to previous proofs of generic sufficiency bounds
for frames \cite{BCE:06, CEHV:15} where an incidence variety is considered.
\subsection{ An affine variety whose real points 
are the space of orthogonal
  projections}\label{sec:background}
\begin{Prop}
There is an affine irreducible subvariety ${\mathcal P}_k(M) \subset \A^{M \times M}$ of complex dimension $k(M-k)$ whose real points are the set of orthogonal projections of rank $k$.
\end{Prop}
\begin{Remark}
  It is crucial for our proof that ${\mathcal P}_k(M)$ be irreducible
  since will need to know that any proper subvariety has strictly
  smaller dimension.
\end{Remark}

% \begin{comment}
% To prove our theorem we need to understand the algebraic structure on
% the space of orthogonal projections of a given rank. Precisely, the
% space of rank $k$ projections in $\R^M$ is the orbit by the action of
% orthogonal group $SO(M)$ on the standard rank $k$ projection $E_k=
% \diag(1,\ldots 1, 0, \ldots , 0)$ given by $A \cdot E_k = AE_k
% A^t$. Since the orthogonal group is a compact real Lie group, the
% space of projections is compact manifold isomorphic to the real
% Grassmannian $G(k,M)$ However, to prove our theorem we view $SO(M)$ as
% the real points of an affine algebraic group. From this we see that in fact the space of
% projections is the real points of the 
% orbit of a point of the affine space $\A^{M \times
%   M}$ parametrizing $M \times M$ matrices under the action of the
% affine algebraic group $SO(M,\C)$. As a consequence it follows that this
% space is the set of real points of an {\em affine} variety.
% \end{comment}

\begin{proof}
  Let ${\mathcal P}_k(M)$ be the algebraic subset of $\A^{M \times M}$
  defined by the equations $P^2 = P$, $P = P^t$ and $\trace(P) =
  k$. A real matrix satisfies these equations
if and only it is an orthogonal projection. So ${\mathcal P}_k(M)_\R$
is the set of orthogonal projections.

We now show that ${\mathcal P}_k(M)$ is an irreducible variety
of dimension $k(M-k)$. 

Let $P$ be a matrix representing a point of ${\mathcal P}_k(M)$.
Since $P^2
  =P$ the eigenvalues of $P$ lie in the set $\{0,1\}$ and $P$ is diagonalizable
. Thus $P$ is a symmetric and diagonalizable\footnote{
Note that a complex symmetric matrix need not be diagonalizable. For
example the matrix $$\left(\begin{array}{cc} 1 & i \\ i & -1\end{array}
\right)$$ is non-diagonalizable.} matrix. Thus it is conjugate by an element
of the complex orthogonal group $SO(M,\C)$ to a diagonal matrix. Finally the
  condition that $\trace P = k$ implies that $P$ is conjugate to the
  diagonal matrix $E_k= \diag(1,1,\ldots 1, 0, \ldots, 0)$ where there
  are $k$ ones and $M-k$ zeros. Conversely, any matrix
of the form $P = AE_k A^t$ with $A \in SO(M,\C)$ satisfies $P^t =P$, $P^2 =P$
and $\trace P = k$.

Thus ${\mathcal P}_k(M)$ can be identified with the $SO(M,\C)$
orbit of the matrix $E_k$ under the conjugation. Since $SO(M,\C)$ is 
an irreducible algebraic group, so is the orbit. Finally, 
the stabilizer of $E_k$ is isomorphic to the subgroup $SO(k) \times SO(M-k)$.
The dimension of the algebraic group $SO(M,\C)$ is $\binom{M}{2}$.
Thus the dimension of ${\mathcal P}_k(M)$ is $\binom{M}{2} - \binom{k}{2}
- \binom{M-k}{2} = k(M-k)$.
\end{proof}

\subsection{Completion of the Proof of theorem \ref{Thm.projectionchar}}
Since the vectors $P_1x, \ldots , P_Nx$ fail to span $\R^M$ if an only
if there is a non-zero vector $y$ which is orthogonal to each
$P_ix$, a collection $\sA_\Phi$ fails to be injective if and only
there are non-zero vectors $x, y$  such
that 
$$y^tP_1x = y^tP_2x = \ldots = y^t P_Mx = 0.$$

Consider the incidence set of tuples $\{(P_1, \ldots, P_N, x,y) |
y^tP_ix = 0\}$ where $P_i \in {\mathcal P}_{k_i}$ and $x,y \in \C^{M}
\smallsetminus \{0\}$.  Since the equations $y^tP_ix = 0$ are
homogeneous in $x$ and $y$
there is a corresponding incidence set
$$\cI = \cI_{k_1, \ldots , k_N,M} \subset
{\mathcal P}_{k_1} \times \ldots \times {\mathcal P}_{k_N}
\times \P^{N-1} \times \P^{N-1}.$$ 
The real points of the algebraic set $\cI$
parametrize tuples of orthogonal projections and non-zero vectors
$(P_1, \ldots , P_N , x, y)$ such that $P_ix$ is orthogonal to $y$
for each $i$.  By Theorem \ref{Thm.projectionchar} if
$(P_1,\ldots, P_N, x,y) \in \cI_\R$ then the map
$\sA_\Phi$ isn't injective for the collection of projections $\Phi = (P_1, \ldots, P_N)$.

We will show that when $N \geq 2M-1$ the variety $\cI$ contains an
open set of complex dimension less than that of ${\mathcal P}_{k_1}
\times \ldots \times {\mathcal P}_{k_N}$ that contains all of the real
points of $\cI$.  This means that $(\cI)_\R$ has real dimension less
than $\sum_{i=1}^M k_i(M-k_i)$. Hence for generic projections $P_1,
\ldots, P_N$ there are no non-zero real vectors $x, y$ such that
$\langle P_ix,y \rangle =0$ for all $i$. In other words $\sA_\Phi$ is
injective for generic collections of projections $P_1, \ldots , P_N$
with $N \geq 2M-1$.

\begin{Prop} \label{Prop.incidencedim}
There is an open subset of 
$\cI$
which contains $\cI_\R$
and has dimension  $\sum_{i=1}^{N} k_i(M -k_i) + 2M -2 -N$.
In particular if $N \geq 2M-1$
this open set has dimension strictly smaller than $\dim \prod_{i=1}^N {\mathcal P_{k_i}}$.
\end{Prop}
\begin{Remark}
Note that since we do not know that $\cI$ is irreducible
we are not asserting that $\cI$ has dimension
$\sum_{i=1}^{N} k_i(M -k_i) + 2M -2 -N$. Instead, we are proving that the union
of the 
irreducible components of $\cI$ that contain all of the real points has
this dimension.
\end{Remark}

\begin{proof}
We show that the image of the 
projection $p_2 \colon \cI \to \P^{N-1} \times \P^{N-1}$
contains a dense open set $U \subset \P^{M-1} \times \P^{M-1}$
which contains $\P^{M-1}_\R \times \P^{M-1}_R$
such that for each $x,y \in U$
the fiber $p_2^{-1}(x,y)$ is non-empty and has 
dimension $\sum_{i=1}^N (k_i(M-k_i) -1)$.
It follows that the incidence $\cI$ contains an
open set of dimension $\sum_{i=1}^{N} k_i(M -k_i) + 2M -2 -N$
and that this open set contains $\cI_\R$.

Observe that the fiber $p_2^{-1}(x,y)$ 
is the algebraic subset  $$\cI_{x,y} \subset \prod_{i=1}^N {\mathcal P_{k_i}}$$
defined by the linear equations $y^tP_1x =0, \ldots, y^tP_N x=0$.
This algebraic subset is the product $\prod_{i=1}^N (\cI_{x,y})_i$
where $(I_{x,y})_i$ is the algebraic subset of ${\mathcal P_{k_i}}$
defined by the linear equation $y^t P_i x = 0$.
\begin{Lemma} \label{Lem.fiberdim}
For each $k$ with $1 \leq k  \leq M-1$ there is a dense open subset 
 $U_k \subset \P^{M-1} \times \P^{M-1}$ 
containing $\P^{M-1}_\R \times \P^{M-1}_\R$ such that for $x, y \in U_{k}$
the algebraic subset
 ${\mathcal P}_{x,y}$ of ${\mathcal P}_k$ defined by the equation
$y^t P x = 0$ has complex dimension $k(M-k) -1$.
\end{Lemma}
Let $U$ be the intersection of all of the $U_{k_i}$ in $\P^{M-1} \times \P^{M-1}$.
By Lemma \ref{Lem.fiberdim} the inverse image of 
$U$ under the projection $\cI \to \P^{M-1}\times \P^{M-1}$ has
dimension equal to $\sum_{i=1}^{N} k_i(M -k_i) + 2M -2 -N$
and contains all of the real points.
\end{proof}

\begin{proof}[Proof of Lemma \ref{Lem.fiberdim}]
The fiber  ${\mathcal P}_{x,y}$ is defined by a single equation in 
the affine variety ${\mathcal P}_k$. Therefore,  by Krull's Hauptidealsatz 
${\mathcal P}_{x,y}$ has dimension $k(M-k) -1$ unless
the equation $y^t P x$ vanishes identically on ${\mathcal P}_k$
or the equation $y^t Px$ does not vanish at all in which 
case ${\mathcal P}_{x,y}$ is empty.

We first show
that if $x,y$ are non-zero vectors in $\R^M$ we can find $P,Q \in {\mathcal P}_k$ such that $y^t P x = 0$ and $y^t Qx\neq 0$. This implies that ${\mathcal P}_{x,y}$ is non-empty and not all of ${\mathcal P}_{x,y}$.

To find $P$ such that $y^t P x=0$ observe that given any non-zero real vector
$x$ we can find a  linear subspace $L$ of dimension $k < M$ 
which is orthogonal to $x$.
If $P_L$ is the orthogonal projection onto $L$ then 
$P_Lx = 0$ and so $y^T P_L x = 0$ as well.

To find $Q$ such that $y^t Q x \neq 0$
requires more care. Since $x$ and $y$ are real vectors
$\langle x, x \rangle \neq 0$ and $\langle y , y \rangle \neq 0$.
Hence  $\langle x + \lambda y, y \rangle$
and $\langle x + \lambda y, x\rangle$ are non-zero for all but finitely
many 
values of $\lambda$. Choose $\lambda$ such that the above inner
products are non-zero and let $L_1$ be the line spanned by $x + \lambda y$.
Let $Q_{L_1}$ be the orthogonal projection onto this line. Then
$Q_{L_1}x$ is non-zero and parallel to $x + \lambda y$ 
so $y^t Q_{L_1} x = \langle Q_{L_1}x, y \rangle \neq 0$
since we also chose $\lambda$ so that $ x + \lambda y$
 is not orthogonal to $y$.

Now let $L_{k-1}$ be any $(k-1)$-dimensional linear subspace
in the orthogonal complement of the linear subspace spanned
by $x$ and $y$ and let $Q_{L_{k-1}}$ be the orthogonal projection
onto this subspace. Then $Q = Q_{L_1} + Q_{L_{k-1}}$ is
the desired projection.

By the theorem on the dimension of the fibers
applied to the morphism
$$\cI_k = \{(P,x,y)| y^tPx =0\} \subset {\mathcal P}_k
\times \P^{M-1} \times \P^{M-1} \to \P^{M-1} \times \P^{M-1}$$
then there is a dense 
open subset $U_k \subset \P^{M-1} \times \P^{M-1}$
where the dimension of the fiber is constant.
Since
the set of real points of $\P^{M-1} \times \P^{M-1}$ has maximal dimension
it is dense, and therefore $U_k$ contains real points. On the other hand
we showed that the dimension of ${\mathcal P}_{x,y}$ is constant for all real points $(x,y)$ of  $\P^{M-1} \times \P^{M-1}$. Therefore $U_k$ contains
$\P^{M-1}_\R \times \P^{M-1}_\R$.
\end{proof}

\section{The case of fewer measurements} \label{sec:2M-2}
Here we prove that if $M = 2^k +1$ and $N \leq 2M-2$
then for any collection of projections $P_1, \ldots , P_N$
the map $\sA_\Phi$ is not injective. Since we can always add projections
to a collection we may assume that $N = 2M-2$.

By Theorem \ref{Thm.projectionchar} 
the map $\sA_\Phi$ is not injective if and only
there is a pair $(x,y) \in \P^{M-1}_\R \times \P^{M-1}_\R$ 
such that $y^t P_i x  = 0$ for all $i$. The equation
$y^tP_ix =0$ is bihomegenous of degree 1 in $x$ and $y$ e, so
we can consider the subvariety  $Z \subset \P^{M-1} \times \P^{M-1}$
defined by the vanishing of the $2M-2$ bilinear forms $\{y^T P_ix\}_{i=1}^{2M-2}$.
We wish to show that if $M = 2^k+1$ then $Z$ has a real point.

\begin{Lemma}\label{lem.firstreduction}
If $Z$ has a non-empty intersection with diagonal in $\P^{M-1} \times \P^{M-1}$
then $\sA_\Phi$
is not injective. 
\end{Lemma}
\begin{Remark}
Note that Lemma \ref{lem.firstreduction} holds whether or not $M = 2^k +1$.
\end{Remark}
\begin{proof} [Proof of Lemma \ref{lem.firstreduction}]
Let $(z,z)$ be a point of $Z$ on the diagonal. Write $z = x + \sqrt{-1} y$
so the condition  $z^t P_i z =0$ implies that $x^tP_ix - y^tP_iy =0$
and $y^TP_ix = 0$ for all $i$. 
If $x$ and $y$ are both non-zero then $(x,y)$ is a real point of $Z$.
If $x$ or $y$ is 0 then $z$ is either real or pure imaginary. In this case,
either $z$ is a real vector or $\sqrt{-1}z$ is a real vector so 
$(z,z)$ also represents a real point of product $\P^{M-1} \times \P^{M-1}$.
\end{proof}

Now suppose that $Z$ has no real points. Then by Lemma \ref{lem.firstreduction}
$Z$ misses the diagonal Since the equations
$x^tP_i y = 0$ are symmetric in $x$ and $y$, we see that $(x,y) \in Z$
if and only if $(y,x) \in Z$ and $(x,y) \neq (y,x)$. 
Also if $(x,y) \in Z$ is not real
then the complex conjugate $(\overline{x}, \overline{y})$ is also a distinct point of $Z$.
It follows that the degree of the intersection
cycle supported on the variety $Z \subset \P^{N-1} \times \P^{N-1}$
must be divisible by 4. On the other hand by \cite[Examples 13.2, 13.3]{Ful:84}
the degree of the intersection cycle supported on 
$Z$ is $\binom{2M-2}{M-1}$. When $M = 2^k +1$,
Legendre's formula \cite[cf. Proof of Lemma 5.3]{CEHV:15} 
for the highest power of a
prime dividing a factorial shows that $\binom{2M-2}{M-1}$
is not divisible by $4$.

\begin{Remark} If the $P_i$ all have rank one then 
the bilinear equation $y^t P_i x =0$ factors as a product $\langle y, v_i \rangle
 \langle x, v_i\rangle = 0$ where $v_i$ is a unit norm vector generating the line
determined by $v_i$. Since the system of linear equations
$$\langle y, v_1 \rangle= \ldots =\langle y, v_{M-1} \rangle = \langle
x, v_{M} \rangle  = \ldots \langle x, v_{2M-2} \rangle=0$$
has a non-trivial real solution, we obtain another proof that 
the bound $N = 2M-1$ is sharp for rank one projections. 
\end{Remark}

{\bf Acknowledgements.} 
The author is grateful to Pete Casazza for suggesting that he look
at the paper \cite{CCPW:13}.
\bigskip

%\nocite{049.1315cj}
%\bibliography{refs}{}
\bibliographystyle{plain}
\def\cprime{$'$}

\vspace{1cm}

\footnotesize

\noindent Dan Edidin ({\tt edidind@missouri.edu})\\
Department of Mathematics,\\ 
University of Missouri,\\ 
Columbia, Missouri 65211 USA\\  \smallskip

\end{document}